\documentclass{amsart}
\usepackage{amssymb, amsmath, latexsym, mathabx, dsfont,tikz}
\usepackage{multirow}
\usepackage{setspace}
\usepackage{enumerate}
\usepackage{diagbox}
\usepackage{stmaryrd}
\usepackage{makecell}

\renewcommand{\baselinestretch}{\baselinestretch}
\renewcommand{\baselinestretch}{1.1}
\numberwithin{equation}{section}

\newtheorem{thm}{Theorem}[section]
\newtheorem{lem}[thm]{Lemma}

\newtheorem{prop}[thm]{Proposition}

\theoremstyle{definition}

\theoremstyle{remark}
\newtheorem{rmk}[thm]{Remark}

\numberwithin{equation}{section}

\newcommand{\gen}{\text{gen}}

\newcommand{\n}{{\mathbb N}}
\newcommand{\z}{{\mathbb Z}}

\begin{document}
%%%%%%%%%%%%%%%%%%%%%%%%%%%%%%%%%%%%%%%%%%%%%%%%%%%%%%%%%%%%%%%%%%%%%%%%%%%%%
%%%%%%%%%%%%%%%%%%%%%%%%%%%%%%%%%%%%%%%%%%%%%%%%%%%%%%%%%%%%%%%%%%%%%%%%%%%%%
\title[]{Tight universal sums of $m$-gonal numbers}

\author{Jangwon Ju and Mingyu Kim}

\address{Department of Mathematics, University of Ulsan, Ulsan 44610, Korea}
\email{jangwonju@ulsan.ac.kr}

\address{Department of Mathematics, Sungkyunkwan University, Suwon 16419, Korea}
\email{kmg2562@skku.edu}

\thanks{This research of the first author was supported by the National Research Foundation of Korea(NRF) grant funded by the Korea government(MSIT) (NRF-2019R1F1A1064037)}

\thanks{This research of the second author was supported by the National Research Foundation of Korea(NRF) grant funded by the Korea government(MSIT) (NRF-2021R1C1C2010133)}

\subjclass[2020]{11D09} \keywords{sums of polygonal numbers, tight universal}
%%%%%%%%%%%%%%%%%%%%%%%%%%%%%%%%%%%%%%%%%%%%%%%%%%%%%%%%%%%%%%%%%%%%%%%%%%%%%
\begin{abstract}
For a positive integer $n$, the set of all integers greater than or equal to $n$ is denoted by $\mathcal T(n)$.
A sum of generalized $m$-gonal numbers $g$ is called tight $\mathcal T(n)$-universal if the set of all nonzero integers represented by $g$ is equal to $\mathcal T(n)$.
In this article, we prove the existence of a minimal tight $\mathcal T(n)$-universality criterion set for 
a sum of generalized $m$-gonal numbers for any pair $(m,n)$.
To achieve this, we introduce an algorithm giving all candidates for tight $\mathcal T(n)$-universal sums of generalized $m$-gonal numbers for any given pair $(m,n)$.
Furthermore, we provide some experimental results on the classification of tight $\mathcal T(n)$-universal sums of generalized $m$-gonal numbers.
\end{abstract}
%%%%%%%%%%%%%%%%%%%%%%%%%%%%%%%%%%%%%%%%%%%%%%%%%%%%%%%%%%%%%%%%%%%%%%%%%%%%%
\maketitle

%\nocite{*}
%%%%%%%%%%%%%%%%%%%%%%%%%%%%%%%%%%%%%%%%%%%%%%%%%%%%%%%%%%%%%%%%%%%%%%%%%%%%%
%%%%%%%%%%%%%%%%%%%%%%%%%%%%%%%%%%%%%%%%%%%%%%%%%%%%%%%%%%%%%%%%%%%%%%%%%%%%%
\section{Introduction}
%%%%%%%%%%%%%%%%%%%%%%%%%%%%%%%%%%%%%%%%%%%%%%%%%%%%%%%%%%%%%%%%%%%%%%%%%%%%%
%%%%%%%%%%%%%%%%%%%%%%%%%%%%%%%%%%%%%%%%%%%%%%%%%%%%%%%%%%%%%%%%%%%%%%%%%%%%%
A positive definite integral quadratic form
$$
f=f(x_1,x_2,\dots,x_k)=\sum_{1\le i,j\le k}a_{ij}x_ix_j\ \ (a_{ij}=a_{ji}\in \z)
$$
is called {\it universal} if it represents all positive integers.
Lagrange's four-square theorem states that the quaternary quadratic form $x^2+y^2+z^2+w^2$ is universal.
Ramanujan \cite{R} found all diagonal quaternary universal quadratic forms.
In 1993, Conway and Schneeberger announced ``15-Theorem" which says that a (positive definite integral) quadratic form representing all positive integers up to 15 actually represents every positive integer.
Bhargava \cite{B} introduced an algorithm, called the escalation method, which yields the classification of universal quadratic forms (see also \cite{C}).
The escalation method shows that if an integral quadratic form $f$ represents nine integers 1,2,3,5,6,7,10,14 and 15, then $f$ is universal.
Kim, Kim and Oh \cite{KKO} generalized this result and proved that for any infinite set $S$ of quadratic forms of bounded rank, there is a finite subset $S_0$ of $S$ such that any (positive definite integral) quadratic form representing every form in $S_0$ represents all of $S$.
Following \cite{KLO}, we call such a set $S_0$ {\it an $S$-universality criterion set}.
An $S$-universality criterion set $S_0$ is called {\it minimal} if no proper subset $S'_0$ of $S_0$ is an $S$-universality criterion set.

For an integer $m\ge 3$, we define a polynomial $P_m(x)$ by
$$
P_m(x)=\frac{(m-2)x^2-(m-4)x}{2}.
$$
An integer of the form $P_m(u)$ for some integer $u$ is called a generalized $m$-gonal number.
A polynomial of the form
$$
a_1P_m(x_1)+a_2P_m(x_2)+\cdots+a_kP_m(x_k)
$$
with positive integers $a_1,a_2,\dots,a_k$ is called {\it a sum of generalized $m$-gonal numbers} or {\it an $m$-gonal form}.
In \cite{KL}, Kane and Liu proved that there is a constant $\gamma_m$ such that if a sum of generalized $m$-gonal numbers represents all positive integers up to $\gamma_m$, then it represents all positive integers.
By applying escalation method to sums of generalized $m$-gonal numbers, they showed the existence of such $\gamma_m$ and found an asymptotic upper bound of $\gamma_m$ in a function of $m$.

In this article, we introduce an algorithm giving all tight $\mathcal T(n)$-universal $m$-gonal forms and provide some experimental results of the algorithm.
In Section 2, some basic notation and terminologies will be given.
In Section 3, we introduce an algorithm which gives the classification of tight $\mathcal T(n)$-universal $m$-gonal forms for each given pair $(m,n)$.
Note that this algorithm is analogous to the escalation algorithm described by Bhargava, and when $n=1$, it coincides with the algorithm for universal $m$-gonal forms appeared in \cite{KL}.
In Section 4, we provide some experimental results on the algorithm described in Section 3 including candidates of tight $\mathcal T(n)$-universal $m$-gonal forms for $m=7,9,10$ and 11.

%%%%%%%%%%%%%%%%%%%%%%%%%%%%%%%%%%%%%%%%%%%%%%%%%%%%%%%%%%%%%%%%%%%%%%%%%%%%%
%%%%%%%%%%%%%%%%%%%%%%%%%%%%%%%%%%%%%%%%%%%%%%%%%%%%%%%%%%%%%%%%%%%%%%%%%%%%%
\section{Preliminaries}
%%%%%%%%%%%%%%%%%%%%%%%%%%%%%%%%%%%%%%%%%%%%%%%%%%%%%%%%%%%%%%%%%%%%%%%%%%%%%
%%%%%%%%%%%%%%%%%%%%%%%%%%%%%%%%%%%%%%%%%%%%%%%%%%%%%%%%%%%%%%%%%%%%%%%%%%%%%
For $k=1,2,3,\dots$, we define a set $\mathcal{N}(k)$ to be the set of all vectors of positive integers with length $k$ and coefficients in the ascending order, i.e.,
$$
\mathcal{N}(k)=\{ \mathbf{a}=(a_1,a_2,\dots,a_k)\in \n^k : a_1\le a_2\le \cdots \le a_k\}.
$$
Put $\mathcal{N}=\bigcup_{k=1}^{\infty}\mathcal{N}(k)$.
For two vectors $\mathbf{a}\in \mathcal{N}(k)$ and $\mathbf{b}\in \mathcal{N}(s)$ with $k\le s$, we write
$$
\mathbf{a}\preceq \mathbf{b}\ (\mathbf{a} \prec \mathbf{b})
$$
if the sequence $(a_i)_{1\le i\le k}$ is a (proper) subsequence of $(b_j)_{1\le j\le s}$, where
$$
\mathbf{a}=(a_1,a_2,\dots,a_k)\ \ \text{and}\ \ \mathbf{b}=(b_1,b_2,\dots,b_s).
$$
Given a vector $\mathbf{a}\in \mathcal{N}(k)$ and a positive integer $a$, we define a vector $\mathbf{a}*a$ by
$$
\mathbf{a}*a=(a_1,a_2,\dots,a_i,a,a_{i+1},a_{i+2},\dots,a_k)\in \mathcal{N}(k+1),
$$
where $i$ is the maximum index satisfying $a_i\le a$, i.e., $\mathbf{a}*a$ is the vector in $\mathcal{N}(k+1)$ with coefficients $a_1,a_2,\dots,a_k$ and $a$.
For $\mathbf{a}\in \mathcal{N}(k)$ and $\mathbf{b}=(b_1,b_2,\dots,b_s)\in \mathcal{N}(s)$, we define $\mathbf{a}*\mathbf{b}$ to be the vector
$$
\mathbf{a}*b_1*b_2*\cdots*b_s \in \mathcal{N}(k+s).
$$
We identify $\mathcal{N}(1)$ with $\n$, e.g., $3*7*2*5$ denotes the vector $(2,3,5,7)\in \mathcal{N}(4)$.
Let $S$ be a set of nonnegative integers containing 0 and 1, and let $n$ be a positive integer.
For a vector $\mathbf{a}=(a_1,a_2,\dots,a_k)\in \mathcal{N}(k)$, we define
$$
R_S(\mathbf{a})=\{ a_1s_1+a_2s_2+\cdots+a_ks_k : s_i\in S\} \ \ \text{and}\ \ R'_S(\mathbf{a})=R_S(\mathbf{a})-\{0\}.
$$
Let $\mathcal{GP}_m$ be the set of generalized $m$-gonal numbers, i.e.,
$$
\mathcal{GP}_m=\left\{ P_m(u) : u\in \z \right\}.
$$
Then an $m$-gonal form
$$
a_1P_m(x_1)+a_2P_m(x_2)+\cdots+a_kP_m(x_k)\ \ (a_1\le a_2\le \cdots \le a_k)
$$
corresponds to the pair $(\mathcal{GP}_m,\mathbf{a})$, where $\mathbf{a}=(a_1,a_2,\dots,a_k)\in \mathcal{N}(k)$.
In this article, a pair $(\mathcal{GP}_m,\mathbf{a})$ ($\mathbf{a}\in \mathcal{N}(k)$) will also be called a $k$-ary $m$-gonal form.
Let $n$ be a positive integer.
An $m$-gonal form $(\mathcal{GP}_m,\mathbf{a})$ is called {\it $\mathcal T(n)$-universal} if $R'_{\mathcal{GP}_m}(\mathbf{a})\supseteq \mathcal T(n)$, and {\it tight $\mathcal T(n)$-universal} if $R'_{\mathcal{GP}_m}(\mathbf{a})=\mathcal T(n)$.
A tight $\mathcal T(n)$-universal $m$-gonal form $(\mathcal{GP}_m,\mathbf{a})$ is called {\it new} if $R'_{\mathcal{GP}_m}(\mathbf{b})\subsetneq \mathcal T(n)$ for every vector $\mathbf{b}\in \mathcal{N}$ satisfying $\mathbf{b}\prec \mathbf{a}$.
When $n=1$, we use the expression ``universal" along with ``tight $\mathcal T(1)$-universal" to follow the convention.

%%%%%%%%%%%%%%%%%%%%%%%%%%%%%%%%%%%%%%%%%%%%%%%%%%%%%%%%%%%%%%%%%%%%%%%%%%%%%
\begin{lem} \label{lemexist}
Let $m$ be an integer greater than or equal to 3 and $n$ be a positive integer.
Then there exists a vector $\mathbf{a}$ such that $R'_{\mathcal{GP}_m}(\mathbf{a})=\mathcal T(n)$.
\end{lem}

\begin{proof}
Let $\mathbf{b}=(n,n,\dots,n)\in \mathcal{N}(m)$ be the vector of length $m$ with every coefficient equal to $n$.
By Fermat polygonal number theorem, we have
$$
R_{\mathcal{GP}_m}(\mathbf{b})=\{ nu : u\in \z_{\ge 0} \}.
$$
From this, one may easily deduce that
$$
R'_{\mathcal{GP}_m}(\mathbf{b}*(n+1)*(n+2)*\cdots*(2n-1))=\mathcal T(n).
$$
This completes the proof.
\end{proof}

%%%%%%%%%%%%%%%%%%%%%%%%%%%%%%%%%%%%%%%%%%%%%%%%%%%%%%%%%%%%%%%%%%%%%%%%%%%%%
%%%%%%%%%%%%%%%%%%%%%%%%%%%%%%%%%%%%%%%%%%%%%%%%%%%%%%%%%%%%%%%%%%%%%%%%%%%%%
\section{An algorithm for tight $\mathcal T(n)$-universal sums of $m$-gonal numbers}
%%%%%%%%%%%%%%%%%%%%%%%%%%%%%%%%%%%%%%%%%%%%%%%%%%%%%%%%%%%%%%%%%%%%%%%%%%%%%
%%%%%%%%%%%%%%%%%%%%%%%%%%%%%%%%%%%%%%%%%%%%%%%%%%%%%%%%%%%%%%%%%%%%%%%%%%%%%
We introduce an algorithm which gives all new tight $\mathcal T(n)$-universal $m$-gonal forms.
Let $m$ be an integer $\ge 3$ and $n$ be a positive integer.
For $\mathbf{a}\in \mathcal{N}$, we denote by $\Psi(\mathbf{a})$ the set of integers in $\mathcal T(n)$ which is not represented by the $m$-gonal form $(\mathcal{GP}_m,\mathbf{a})$, i.e.,
$$
\Psi(\mathbf{a})=\Psi_{m,n}(\mathbf{a})=\mathcal T(n)-R'_{\mathcal{GP}_m}(\mathbf{a}).
$$
We define a function $\psi=\psi_{m,n} : \mathcal{N}\to \mathcal T(n)\cup \{ \infty \}$ by
$$
\psi(\mathbf{a})=\begin{cases}\min(\Psi(\mathbf{a}))&\text{if}\ \ \Psi(\mathbf{a})\neq \emptyset,\\
\infty&\text{otherwise.}\end{cases}
$$
For a vector $\mathbf{a}$ with $\psi(\mathbf{a})<\infty$, we define a set $\mathcal E(\mathbf{a})$ to be
$$
\mathcal E(\mathbf{a})=\{ g\in \z : n\le g\le \psi(\mathbf{a})-n\} \cup \{ \psi(\mathbf{a})\}.
$$
Note that if $\psi(\mathbf{a})<2n$, then $\mathcal E(\mathbf{a})=\{ \psi(\mathbf{a})\}$.
For $k=1,2,3,\dots$, we define subsets $E(k),U(k),NU(k)$ and $A(k)$ of $\mathcal{N}(k)$ recursively as follow;
Put $E(1)=\{(n)\}$.
Define 
$$
U(k)=\{ \mathbf{a}\in E(k) : \psi(\mathbf{a})=\infty \}.
$$
Let $NU(k)$ to be the set of all vectors $\mathbf{a}$ in $U(k)$ such that $\mathbf{b}\not\in \bigcup_{i=1}^{k-1}U(i)$ for every $\mathbf{b}\in \mathcal{N}$ satisfying $\mathbf{b}\prec \mathbf{a}$.
Let $A(k)=E(k)-U(k)$ and
$$
E(k+1)=\bigcup_{\mathbf{a}\in A(k)}\left\{ \mathbf{a}*g : g\in \mathcal E(\mathbf{a})\right\}.
$$
The algorithm terminates once $A(k)=\emptyset$.

%%%%%%%%%%%%%%%%%%%%%%%%%%%%%%%%%%%%%%%%%%%%%%%%%%%%%%%%%%%%%%%%%%%%%%%%%%%%%
\begin{thm}
Under the notations given above,
for a vector $\mathbf{a}\in \mathcal{N}(k)$, a $k$-ary $m$-gonal form $(\mathcal{GP}_m,\mathbf{a})$ is new tight $\mathcal T(n)$-universal if and only if $\mathbf{a}\in NU(k)$.
\end{thm}

\begin{proof}
Note that ``if'' part is clear by construction.
To prove ``only if" part, let $\mathbf{a}\in \mathcal{N}(k)$ be a vector such that $(\mathcal{GP}_m,\mathbf{a})$ is tight $\mathcal T(n)$-universal.
Since $R'_{\mathcal{GP}_m}(\mathbf{a})=\mathcal T(n)$, it clearly follows that $a_{i_1}=n$, where we put $i_1=1$.
Note that the set $R_{\mathcal{GP}_m}(a_{i_1})$ does not contain any positive integer less than $n$ and it does contain 0 and all integers from $n$ to $\psi(a_{i_1})-1$.
From this and $\psi(a_{i_1})\in \mathcal T(n)=R'_{\mathcal{GP}_m}(\mathbf{a})$, one may easily deduce that there must be an index $i_2$ different from $i_1$ such that
$$
a_{i_2}\in \mathcal E(a_{i_1})=\{ n,n+1,n+2,\dots,\psi(a_{i_1})-n\} \cup \{\psi(a_{i_1})\}.
$$
Thus we have $a_{i_1}*a_{i_2}\preceq \mathbf{a}$, where $a_{i_1}*a_{i_2}\in E(2)$.
Note that $\psi(a_{i_1})\in R'_{\mathcal{GP}_m}(a_{i_1}*a_{i_2})$.
Assume $R'_S(a_{i_1}*a_{i_2})\subsetneq \mathcal T(n)$ so that $\psi(a_{i_1}*a_{i_2})<\infty$.
One may easily show that there should be an index $i_3$ different from both $i_1$ and $i_2$ such that
$$
a_{i_3}\in \mathcal E(a_{i_1}*a_{i_2})=\{ n,n+1,n+2,\dots,\psi(a_{i_1}*a_{i_2})-n\} \cup \{ \psi(a_{i_1}*a_{i_2})\}
$$
in a similar manner.
We have $a_{i_1}*a_{i_2}*a_{i_3}\in E(3)$ by construction.
Note that
$$
\psi(a_{i_1}*a_{i_2}*\cdots*a_{i_j})<\infty
$$
for every $j=1,2,\dots,k-1$ since otherwise $(\mathcal{GP}_m,\mathbf{a})$ cannot be new.
Repeating this, we arrive at
$$
\mathbf{a}=a_{i_1}*a_{i_2}*\cdots*a_{i_k}\in E(k).
$$
Since $(\mathcal{GP}_m,\mathbf{a})$ is new tight $\mathcal T(n)$-universal, one may easily see that $\mathbf{a}\in NU(k)$.
This completes the proof.
\end{proof}

%%%%%%%%%%%%%%%%%%%%%%%%%%%%%%%%%%%%%%%%%%%%%%%%%%%%%%%%%%%%%%%%%%%%%%%%%%%%%

Though the proof of the following theorem appeared in the proof of \cite[Lemma 2.1]{KL}, we provide it for completeness.
For two positive integers $d$ and $r$, we define a set
$$
\mathcal {AP}_{d,r}=\{dg+r : g\in \n \cup \{0\} \} \ (\subseteq \n).
$$

%%%%%%%%%%%%%%%%%%%%%%%%%%%%%%%%%%%%%%%%%%%%%%%%%%%%%%%%%%%%%%%%%%%%%%%%%%%%%
\begin{lem} \label{lemesc}
Under the same notations given above, there is a positive integer $l=l(m,n)$ depending on $m$ and $n$ such that $A(l)=\emptyset$.
\end{lem}

\begin{proof}
Let $t$ be a positive integer greater than 4 and let $\mathbf{a}=(a_1,a_2,\dots,a_t)$ be a vector in $A(t)=E(t)-U(t)$ so that $\psi(\mathbf{a})<\infty$.
Note that, for any $\z$-lattice $L$ of rank $\ge 4$ with $Q(\gen(L))\subsetneq \n$,
$$
\n-Q(\gen(L))=\bigcup_{i=1}^{\nu'_1} \mathcal {AP}_{d'_i,r'_i}
$$
for some positive integers $\nu'_1,d'_i$ and $r'_i$ with $r'_i<d'_i$ by the results in \cite{OM2}.
From this and \cite[Theorem 4.9]{CO}, one may easily deduce that
$$
\mathcal T(n)-R'_{\mathcal{GP}_m}(\mathbf{a})=\bigcup_{i=1}^{\nu_1} \mathcal {AP}_{d_i,r_i} \cup \{ e_1,e_2,\dots,e_{\nu_2}\}
$$
for some nonnegative integers $\nu_1,\nu_2$ not both 0 and some positive integers $d_i,r_i,e_j$ with $e_j\not\in \bigcup_{i=1}^{\nu_1} \mathcal {AP}_{d_i,r_i}$ for all $j=1,2,\dots,\nu_2$.
Suppose that $g_1$ is a positive integer with $n\le g_1\le \psi(\mathbf{a})-n$ or $g_1=\psi(\mathbf{a})$ so that $\mathbf{a}*g_1\in E(t+1)$.
If
$$
Q(\gen(\langle a_1,a_2,\dots,a_t\rangle))\subsetneq Q(\gen(\langle a_1,a_2,\dots,a_t,g_1\rangle)),
$$
then
$$
\mathcal T(n)-R'_{\mathcal{GP}_m}(\mathbf{a}*g_1)=\bigcup_{w=1}^{\nu_3} \mathcal {AP}_{d_{i_w},r_{i_w}} \cup \{ e'_1,e'_2,\dots,e'_{\nu_4}\},
$$
where $\nu_3$ is an integer with
$$
0\le \nu_3<\nu_1,\quad (i_1,i_2,\dots,i_{\nu_3})\prec (1,2,\dots,\nu_1),
$$
and $\nu_4$ a nonnegative integer.
When
$$
Q(\gen(\langle a_1,a_2,\dots,a_t\rangle))=Q(\gen(\langle a_1,a_2,\dots,a_t,g_1\rangle)),
$$
then it follows that
$$
\mathcal T(n)-R'_{\mathcal{GP}_m}(\mathbf{a}*g_1)=\bigcup_{i=1}^{\nu_1} \mathcal {AP}_{d_i,r_i} \cup \{ e_{j_1},e_{j_2},\dots,e_{j_{\nu_5}}\},
$$
where $\nu_5$ is a nonnegative integer less than $\nu_2$ and $(j_1,j_2,\dots,j_{\nu_5})\prec (1,2,\dots,\nu_2)$.

Let $\mathbf{b}$ be a vector in $A(5)=E(5)-U(5)$.
From what we observed above, we may define a positive integer $w=w(\mathbf{b})$ to be the maximal positive integer $w$ satisfying
$$
b*g_1*g_2*\cdots*g_i\in A(5+i)-U(5+i),\ g_i\in \mathcal E(b*g_1*g_2*\cdots*g_{i-1}),
$$
for every $i=1,2,\dots,w-1$.
Since the set $E(5)$ is finite by construction, we may take $l$ as
$$
l=5+\max \{w(\mathbf{b}) : \mathbf{b}\in E(5)-U(5)\}.
$$
This completes the proof.
\end{proof}

%%%%%%%%%%%%%%%%%%%%%%%%%%%%%%%%%%%%%%%%%%%%%%%%%%%%%%%%%%%%%%%%%%%%%%%%%%%%%
We now introduce our main result which gives a natural generalization of Conway-Schneeber 15-Theorem to the case of tight $\mathcal T(n)$-universal $m$-gonal forms.

%%%%%%%%%%%%%%%%%%%%%%%%%%%%%%%%%%%%%%%%%%%%%%%%%%%%%%%%%%%%%%%%%%%%%%%%%%%%%
\begin{thm} \label{thmcv}
Under the same notations given above, there is a finite set $CS(m,n)$ such that $R'_{\mathcal{GP}_m}(\mathbf{a})=\mathcal T(n)$ if and only if $R'_{\mathcal{GP}_m}(\mathbf{a})\cap \{1,2,\dots,n-1\}=\emptyset$ and $CS(m,n)\subset R'_{\mathcal{GP}_m}(\mathbf{a})$ for any vector $\mathbf{a}\in \mathcal{N}$.
\end{thm}

\begin{proof}
Using Lemma \ref{lemesc}, we take the smallest positive integer $l$ satisfying $A(l)=\emptyset$.
Define a finite set
$$
CS(m,n)=\{n\} \cup \bigcup_{k=1}^{l-1}\{ \psi(\mathbf{a}) : \mathbf{a}\in A(k)\}.
$$
Let $\mathbf{a}\in \mathcal{N}$ be a vector with $R'_{\mathcal{GP}_m}(\mathbf{a})\cap \{1,2,\dots,n-1\}=\emptyset$ such that $R'_{\mathcal{GP}_m}(\mathbf{a})\supset CS(m,n)$.
From the condition that $R'_{\mathcal{GP}_m}(\mathbf{a})\supset CS(m,n)$, one may easily see that there is a vector $\mathbf{b}\in \mathcal{N}$ with $\mathbf{b}\preceq \mathbf{a}$ such that $\mathbf{b}\in U(k)$ for some $k$ less than or equal to $l=l(m,n)$.
It follows that
$$
\mathcal T(n)=R'_{\mathcal{GP}_m}(\mathbf{b})\subseteq R'_{\mathcal{GP}_m}(\mathbf{a}).
$$
This completes the proof.
\end{proof}

%%%%%%%%%%%%%%%%%%%%%%%%%%%%%%%%%%%%%%%%%%%%%%%%%%%%%%%%%%%%%%%%%%%%%%%%%%%%%
\begin{rmk} \label{rmkmin}
In Theorem \ref{thmcv}, the set $CS(m,n)$ is minimal in the sense that for any $g\in CS(m,n)$, there is a vector $\mathbf{b}\in \mathcal{N}$ such that $R'_{\mathcal{GP}_m}(\mathbf{b})=\mathcal T(n)-\{g\}$.
To see this, we take $\mathbf{b}=\mathbf{c}*\mathbf{d}$, where $\psi(\mathbf{c})=g$ and $R'_{\mathcal{GP}_m}(\mathbf{d})=\mathcal T(g+1)$.
The existence of such vectors $\mathbf{c}$ and $\mathbf{d}$ follows from the definition of the set $CS(m,n)$ and Lemma \ref{lemexist}, respectively.
\end{rmk}

%%%%%%%%%%%%%%%%%%%%%%%%%%%%%%%%%%%%%%%%%%%%%%%%%%%%%%%%%%%%%%%%%%%%%%%%%%%%%
In the spirit of Remark \ref{rmkmin} and \cite{KLO}, we may call the set $CS(m,n)$ {\it a minimal tight $\mathcal T(n)$-universality criterion set for $m$-gonal forms}.

%%%%%%%%%%%%%%%%%%%%%%%%%%%%%%%%%%%%%%%%%%%%%%%%%%%%%%%%%%%%%%%%%%%%%%%%%%%%%
\begin{prop} \label{propbase}
Let $m$ be an integer greater than or equal to 3 different from 5 and let $n$ be in integer greater than 1.
Under the notations given above, we have the following;
\begin{enumerate} [(i)]
\item $\{n,n+1,n+2,\dots,2n\} \subseteq CS(m,n)$;
\item $E(k)=\{(n,n+1,n+2,\dots,n+k-1)\}$ for $k=1,2,\dots,n$;
\item $U(k)=\emptyset$ $(\text{or equivalently},\ A(k)=E(k))$ for $k=1,2,\dots,n$;
\item $E(n+1)=\{(n,n,n+1,n+2,\dots,2n-1),(n,n+1,n+2,\dots,2n-1,2n)\}$.
\end{enumerate}
\end{prop}

\begin{proof}
Note that $2\not\in \mathcal{GP}_m$ since $m\neq 5$.
For $i=1,2,\dots,n-1$, one may easily show that $\psi(n)=n+1$ and
$$
\psi(n,n+1,n+2,\dots,n+i)=n+i+1.
$$
The lemma follows directly from this.
\end{proof}

%%%%%%%%%%%%%%%%%%%%%%%%%%%%%%%%%%%%%%%%%%%%%%%%%%%%%%%%%%%%%%%%%%%%%%%%%%%%%
\begin{rmk}
Note that Proposition \ref{propbase}(i),(ii),(iii) also hold for the case of pentagonal forms, i.e., when $m=5$.
However, Proposition \ref{propbase}(iv) is no longer true when $m=5$.
In fact, since $2=P_5(-1)\in \mathcal{GP}_5$, we have
$$
2n\in R'_{\mathcal{GP}_5}(n)\subset R'_{\mathcal{GP}_5}(n,n+1,n+2,\dots,2n-1),
$$
and thus we would have $\psi(n,n+1,n+2,\dots,2n-1)>2n$.
\end{rmk}

%%%%%%%%%%%%%%%%%%%%%%%%%%%%%%%%%%%%%%%%%%%%%%%%%%%%%%%%%%%%%%%%%%%%%%%%%%%%%
%%%%%%%%%%%%%%%%%%%%%%%%%%%%%%%%%%%%%%%%%%%%%%%%%%%%%%%%%%%%%%%%%%%%%%%%%%%%%
\section{Some experimental results}
%%%%%%%%%%%%%%%%%%%%%%%%%%%%%%%%%%%%%%%%%%%%%%%%%%%%%%%%%%%%%%%%%%%%%%%%%%%%%
%%%%%%%%%%%%%%%%%%%%%%%%%%%%%%%%%%%%%%%%%%%%%%%%%%%%%%%%%%%%%%%%%%%%%%%%%%%%%
We provide some experimental results on the escalation algorithm for tight $\mathcal T(n)$-universal $m$-gonal forms.
We first note that, in practice, we use the set
$$
\Psi(\mathbf{a})=\Psi_{m,n}(\mathbf{a})=\{ u \in \mathcal T(n) : u\le 10^6 \}-R'_{\mathcal{GP}_m}(\mathbf{a})
$$
instead of the original definition $\Psi(\mathbf{a})=\mathcal T(n)-R'_{\mathcal{GP}_m}(\mathbf{a})$ in the algorithm so that
$$
\{ u\in \n : n\le u\le 10^6\} \subset R'_{\mathcal{GP}_m}(\mathbf{a}),\ \ \forall \mathbf{a}\in \bigcup_{k=1}^{\infty} U(k).
$$
For any integer $m\ge 3$ and a positive integer $n$, we define $\gamma_{m,n}$ to be the maximum element in the set $CS(m,n)$ which was defined in the proof of Theorem \ref{thmcv}.
By Theorem \ref{thmcv}, if an $m$-gonal form $g$ does not represent any integer less than $n$ and does represent all integers from $n$ to $\gamma_{m,n}$, then $g$ is tight $\mathcal T(n)$-universal.

In Table \ref{tablecmn}, the sets $CS(m,n)$ will be given for some pairs $(m,n)$.
In the table, the pairs $(m,n)$ are marked with $\dag$ when the tight $\mathcal T(n)$-universal $m$-gonal forms are already completely classified so that the set $CS(m,n)$ in the table has been proved to be equal to the set $CS(m,n)$ in the algorithm in Secion 3.

For the classification of tight $\mathcal T(n)$-universal $m$-gonal forms, we refer the reader to \cite{B} for $(m,n)=(4,1)$, \cite{BK} for $(m,n)=(3,1)$, \cite{JO} for $(m,n)=(8,1)$, \cite{J} for $(m,n)=(5,1)$, \cite{KO} for $m=4$ and $n\ge 2$, and \cite{MG} for the others.
Note that universal $m$-gonal forms were classified in chronological order of $m=4,3,8,5$.
The tight universal $m$-gonal forms are classified for $m=4,3$, and tight $\mathcal T(n)$-universal octagonal forms for all $n\ge 2$ are treated in the paper \cite{JK}.
In this spirit, we provide the candidates for tight $\mathcal T(n)$-universal pentagonal forms in the cases of $n=2,3$ in Tables \ref{table52}, \ref{table53}, respectively.
Note that there is exactly one candidate for tight $\mathcal T(n)$-universal pentagonal forms for each $n=4,5,6$, which is $(\mathcal{GP}_5,(n,n+1,n+2,\dots,2n-1))$.

Now we give attention to universal $m$-gonal forms.
For $m=3,4,\dots$, we define
$$
\gamma_m=\gamma_{m,1}=\max(C(m,1)).
$$
In Table \ref{tablecm1}, $\gamma_m$ is given for $3\le m\le 11$ and the proved cases are marked with $\dag$.
We provide all candidates of new universal $m$-gonal forms, for $m=7,9,10,11$, in Tables \ref{table71}-\ref{table111}, since the universal $m$-gonal forms are of particular interest to many mathematicians.

%%%%%%%%%%%%%%%%%%%%%%%%%%%%%%%%%%%%%%%%%%%%%%%%%%%%%%%%%%%%%%%%%%%%%%%%%%%%%
\begin{table}[h]
\caption{$\gamma_m$ for $3\le m\le 11$}
\vskip -7pt
\label{tablecm1}
\begin{center}
\begin{tabular}{|c|c|c|c|c|c|c|c|c|}
\hline
$m$ & $3^{\dag}$ & $4^{\dag}$ & $5^{\dag}$ & 7 & $8^{\dag}$ & 9 & 10 & 11\\
\hline
$\gamma_m$&8&15&109&131&60&69&46&45\\
\hline
\end{tabular}
\end{center}
\end{table}
\vskip -13pt

%%%%%%%%%%%%%%%%%%%%%%%%%%%%%%%%%%%%%%%%%%%%%%%%%%%%%%%%%%%%%%%%%%%%%%%%%%%%%
\begin{table}[h]
\caption{Candidates for new tight $\mathcal T(2)$-universal pentagonal forms $(\mathcal{GP}_5,(a_1,a_2,\dots,a_k))$}
\vskip -8pt
\label{table52}
\begin{center}
\begin{tabular}{|cccc|c|}
\hline
$a_1$ & $a_2$ & $a_3$ & $a_4$ & Conditions on $a_k$ ($3\le k\le 4$)\\
\hline
2 & 2 & 3 & &\\
2 & 3 & $a_3$ & & $6\le a_3\le 9$, $a_3\neq 8$\\
2 & 3 & 3 & $a_4$ & $3\le a_4\le 77$, $a_4\neq 6,7,9,76$\\
2 & 3 & 4 & $a_4$ & $4\le a_4\le 141$, $a_4\neq 6,7,9,140$\\
2 & 3 & 5 & $a_4$ & $5\le a_4\le 53$, $a_4\neq 6,7,9,52$\\
\hline
\end{tabular}
\end{center}
\end{table}
\vskip -13pt

%%%%%%%%%%%%%%%%%%%%%%%%%%%%%%%%%%%%%%%%%%%%%%%%%%%%%%%%%%%%%%%%%%%%%%%%%%%%%
\begin{table}[h]
\caption{Candidates for new tight $\mathcal T(3)$-universal pentagonal forms $(\mathcal{GP}_5,(a_1,a_2,\dots,a_k))$}
\vskip -8pt
\label{table53}
\begin{center}
\begin{tabular}{|ccccc|c|}
\hline
$a_1$ & $a_2$ & $a_3$ & $a_4$ & $a_5$ & Conditions on $a_k$ ($4\le k\le 5$)\\
\hline
3 & 3 & 4 & 5 & & \\
3 & 4 & 4 & 5 & & \\
3 & 4 & 5 & $a_4$ & & $6\le a_4\le 22$, $a_4\neq 10,15,20,21$\\
\hline
3 & 4 & 5 & 5 & $a_5$ & $a_5=5,10,15,20,21,62$ or $23\le a_5\le 59$\\
3 & 4 & 5 & 10 & $a_5$ & $a_5=10,15,20,21,47$ or $23\le a_5\le 44$\\
3 & 4 & 5 & 15 & $a_5$ & $a_5=15,20,21,52$ or $23\le a_5\le 49$\\
\hline
\end{tabular}
\end{center}
\end{table}

%%%%%%%%%%%%%%%%%%%%%%%%%%%%%%%%%%%%%%%%%%%%%%%%%%%%%%%%%%%%%%%%%%%%%%%%%%%%%
\begin{table}[h]
\caption{$CS(m,n)$ for some pairs $(m,n)$}
\vskip -7pt
\label{tablecmn}
\renewcommand{\arraystretch}{1.2}
\begin{center}
\begin{tabular}{|c|c|l|}
\hline
$m$ & $n$ & $CS(m,n)$\\
\hline
\multirow{4}{*}{3}
& $1^{\dag}$ & $\{1,2,4,5,8\}$\\ \cline{2-3}
& $2^{\dag}$ & $\{2,3,4,8,10,16,19\}$\\ \cline{2-3}
& $3^{\dag}$ & $\{3,4,5,6,16\}$\\ \cline{2-3}
& $\ge 4^{\dag}$ & $\{n,n+1,n+2,\dots,2n\}$\\
\hline
\multirow{4}{*}{4} 
& $1^{\dag}$ & $\{1,2,3,5,6,7,10,14,15\}$\\ \cline{2-3}
& $2^{\dag}$ & $\{2,3,4,6,9,10,13,15,17,23\}$\\ \cline{2-3}
& $3^{\dag}$ & $\{3,4,5,6,13,14,18,25,35,46\}$\\ \cline{2-3}
& $\ge 4^{\dag}$ & $\{n,n+1,n+2,\dots,2n\}$\\
\hline
\multirow{5}{*}{5} 
& $1^{\dag}$ & $\{1,3,8,9,11,18,19,25,27,43,98,109\}$\\ \cline{2-3}
& 2 & $\{2,3,9,53,77,141\}$\\  \cline{2-3}
& 3 & $\{3,4,5,22,47,52,62\}$\\ \cline{2-3}
& $4\le n\le 6$ & $\{n,n+1,n+2,\dots,2n-1\}$\\ \cline{2-3}
& $\ge 7^{\dag}$ & $\{n,n+1,n+2,\dots,2n-1\}$\\
\hline
\multirow{5}{*}{7}
& 1 & $\{1,2,3,5,6,9,10,15,16,19,23,31,131\}$\\ \cline{2-3}
& 2 & $\{2,3,4,6,9,10,13,15,18,27,30,32,50\}$\\  \cline{2-3}
& 3 & $\{3,4,5,6,13,14,18\}$\\ \cline{2-3}
& $4\le n\le 10$ & $\{n,n+1,n+2,\dots,2n\}$\\ \cline{2-3}
& $\ge 11^{\dag}$ & $\{n,n+1,n+2,\dots,2n\}$\\
\hline
\multirow{6}{*}{8}
& $1^{\dag}$ & $\{1,2,3,4,6,7,9,12,13,14,18,60\}$\\ \cline{2-3}
& 2 & $\{2,3,4,6,8,9,11,12,14,18\}$\\ \cline{2-3}
& 3 & $\{3,4,5,6,13,14,16,17,21,22,27,36\}$\\ \cline{2-3}
& 4 & $\{4,5,6,7,8,23,28\}$\\ \cline{2-3}
& $5\le n\le 10$ & $\{n,n+1,n+2,\dots,2n\}$\\ \cline{2-3}
& $\ge 11^{\dag}$ & $\{n,n+1,n+2,\dots,2n\}$\\
\hline
\multirow{6}{*}{9}
& 1 & $\{1,2,3,4,5,7,8,10,11,14,16,17,20,22,23,29,32,34,69\}$\\ \cline{2-3}
& 2 & $\{2,3,4,6,8,9,10,11,13,14,16,17,19,23,25,28,34,37,58\}$\\ \cline{2-3}
& 3 & $\{3,4,5,6,13,14,16,17,19,20,21,25,26,28,38,46,53\}$\\ \cline{2-3}
& 4 & $\{4,5,6,7,8,23,25,27,28,32,33\}$\\ \cline{2-3}
& $5\le n\le 12$ & $\{n,n+1,n+2,\dots,2n\}$\\ \cline{2-3}
& $\ge 13^{\dag}$ & $\{n,n+1,n+2,\dots,2n\}$\\
\hline
$\ge 10$ & $\ge 2m-5^{\dag}$ & $\{n,n+1,n+2,\dots,2n\}$\\
\hline
\end{tabular}
\end{center}
\end{table}

%%%%%%%%%%%%%%%%%%%%%%%%%%%%%%%%%%%%%%%%%%%%%%%%%%%%%%%%%%%%%%%%%%%%%%%%%%%%%
\begin{table}[h]
\caption{Candidates for new universal heptagonal forms $(\mathcal{GP}_7,(a_1,a_2,\dots,a_k))$}
\vskip -7pt
\label{table71}
\renewcommand{\arraystretch}{1.1}
\begin{center}
\begin{tabular}{|ccccc|c|}
\hline
$a_1$ & $a_2$ & $a_3$ & $a_4$ & $a_5$ & Conditions on $a_k$ ($4\le k\le 5$)\\
\hline
1 & 1  & 1 & $a_4$ & &  $1\le a_4\le 10$, $a_4\neq 6$\\
1 & 1  & 2 & $a_4$ & &  $2\le a_4\le 23$\\
1 & 1  & 3 & $a_4$ & &  $4\le a_4\le 5$\\
1 & 2  & 2 & $a_4$ & &  $2\le a_4\le 19$\\
1 & 2  & 3 & $a_4$ & &  $3\le a_4\le 31$\\
1 & 2  & 4 & $a_4$ & &  $4\le a_4\le 131$\\
1 & 2  & 5 & $a_4$ & &  $5\le a_4\le 10$, $a_4\neq 6$\\
\hline
1 & 1  & 1 & 6 & $a_5$ &  $a_5=6$ or $11\le a_5\le 16$\\
1 & 1  & 3 & 3 & $a_5$ &  $a_5=3$ or $6\le a_5\le 9$\\
1 & 1  & 3 & 6 & $a_5$ & $6\le a_5\le 15$\\
1 & 2  & 5 & 6 & $a_5$ &  $a_5=6$ or $11\le a_5\le 16$\\
\hline
\end{tabular}
\end{center}
\end{table}

%%%%%%%%%%%%%%%%%%%%%%%%%%%%%%%%%%%%%%%%%%%%%%%%%%%%%%%%%%%%%%%%%%%%%%%%%%%%%
\begin{table}[h]
\caption{Candidates for new universal nonagonal forms $(\mathcal{GP}_9,(a_1,a_2,\dots,a_k))$}
\vskip -7pt
\label{table91}
\begin{center}
\begin{tabular}{|ccccccc|c|}
\hline
$a_1$ & $a_2$ & $a_3$ & $a_4$ & $a_5$ & $a_6$ & $a_7$ & Conditions on $a_k$ ($4\le k\le 7$)\\
\hline
1 & 1  & 1 & $a_4$ & & & & $a_4=2,4$\\
1 & 1  & 2 & $a_4$ & & & & $2\le a_4\le 5$\\
1 & 1  & 3 & $a_4$ & & & & $a_4=4,7$\\
1 & 2  & 2 & $a_4$ & & & & $a_4=3,4,7$\\
1 & 2  & 3 & $a_4$ & & & & $a_4=4,5$\\
1 & 2  & 4 & $a_4$ & & & & $4\le a_4\le 12$, $a_4\neq 6,9$\\
\hline
1 & 1  & 1 & 1 & $a_5$ & & & $a_5=1,3,5$\\
1 & 1  & 1 & 3 & $a_5$ & & & $3\le a_5\le 17$, $a_5\neq 4,7$\\
1 & 1  & 3 & 3 & $a_5$ & & & $5\le a_5\le 11$, $a_5\neq 6,7$\\
1 & 1  & 3 & 5 & $a_5$ & & & $5\le a_5\le 16$, $a_5\neq 7$\\
1 & 1  & 3 & 6 & $a_5$ & & & $6\le a_5\le 14$, $a_5\neq 7$\\
1 & 1  & 3 & 8 & $a_5$ & & & $8\le a_5\le 16$\\
1 & 2  & 2 & 2 & $a_5$ & & & $2\le a_5\le 34$, $a_5\neq 3,4,7$\\
1 & 2  & 2 & 5 & $a_5$ & & & $5\le a_5\le 22$, $a_5\neq 7$\\
1 & 2  & 2 & 6 & $a_5$ & & & $6\le a_5\le 22$, $a_5\neq 7$\\
1 & 2  & 3 & 3 & $a_5$ & & & $a_5=3$ or $6\le a_5\le 10$\\
1 & 2  & 3 & 6 & $a_5$ & & & $6\le a_5\le 23$\\
1 & 2  & 3 & 7 & $a_5$ & & & $7\le a_5\le 17$, $a_5\neq 15$\\
1 & 2  & 4 & 6 & $a_5$ & & & $a_5=6,9$ or $13\le a_5\le 20$\\
1 & 2  & 4 & 9 & $a_5$ & & & $a_5=9$ or $13\le a_5\le 29$\\
1 & 2  & 4 & 13 & $a_5$ & & & $13\le a_5\le 69$\\
1 & 2  & 4 & 14 & $a_5$ & & & $14\le a_5\le 34$\\
\hline
1 & 1  & 3 & 3 & 3 & $a_6$ & & $a_6=6$ or $12\le a_6\le 14$\\
1 & 1  & 3 & 3 & 6 & $a_6$ & & $15\le a_6\le 17$\\
1 & 2  & 3 & 7 & 15 & $a_6$ & & $a_6=15$ or $18\le a_6\le 32$\\
\hline
1 & 1  & 3 & 3 & 3 & 3 & $a_7$ & $a_7=3,15,16,17$\\
\hline
\end{tabular}
\end{center}
\end{table}

%%%%%%%%%%%%%%%%%%%%%%%%%%%%%%%%%%%%%%%%%%%%%%%%%%%%%%%%%%%%%%%%%%%%%%%%%%%%%
\begin{table}[h]
\caption{Candidates for new universal decagonal forms $(\mathcal{GP}_{10},(a_1,a_2,\dots,a_k))$}
\label{table101}
\renewcommand{\arraystretch}{1.1}
\begin{center}
\begin{tabular}{|cccccccc|c|}
\hline
$a_1$ & $a_2$ & $a_3$ & $a_4$ & $a_5$ & $a_6$ & $a_7$ & $a_8$ & Conditions on $a_k$ ($4\le k\le 8$)\\
\hline
1 & 1 & 1 & 4 & & & & & \\
1 & 1 & 2 & $a_4$ & & & & & $2\le a_4\le 5$\\
1 & 2 & 2 & $a_4$ & & & & & $3\le a_4\le 4$\\
1 & 2 & 3 & $a_4$ & & & & & $a_4=4,6$\\
1 & 2 & 4 & $a_4$ & & & & & $a_4=4,5,8$\\
\hline
1 & 1 & 1 & 1 & $a_5$ & & & & $a_5=2,3,5$\\
1 & 1 & 1 & 2 & 6 & & & & \\
1 & 1 & 1 & 3 & $a_5$ & & & & $5\le a_5\le 16$\\
1 & 1 & 3 & 3 & $a_5$ & & & & $a_5=5,8$\\
1 & 1 & 3 & 4 & $a_5$ & & & & $4\le a_5\le 16$\\
1 & 1 & 3 & 5 & $a_5$ & & & & $5\le a_5\le 24$\\
1 & 1 & 3 & 6 & $a_5$ & & & & $7\le a_5\le 11$, $a_5\neq 9$\\
1 & 2 & 2 & 2 & $a_5$ & & & & $a_5=2$ or $5\le a_5\le 8$\\
1 & 2 & 2 & 5 & $a_5$ & & & & $6\le a_5\le 13$\\
1 & 2 & 2 & 6 & $a_5$ & & & & $7\le a_5\le 19$, $a_5\neq 14$\\
1 & 2 & 3 & 3 & $a_5$ & & & & $3\le a_5\le 11$, $a_5\neq 4,6,8$\\
1 & 2 & 3 & 5 & $a_5$ & & & & $5\le a_5\le 16$, $a_5\neq 6$\\
1 & 2 & 3 & 7 & $a_5$ & & & & $7\le a_5\le 26$\\
1 & 2 & 3 & 8 & $a_5$ & & & & $8\le a_5\le 16$, $a_5\neq 12,15$\\
1 & 2 & 4 & 6 & $a_5$ & & & & $6\le a_5\le 23$, $a_5\neq 8$\\
1 & 2 & 4 & 7 & $a_5$ & & & & $7\le a_5\le 39$, $a_5\neq 8$\\
\hline
1 & 1 & 1 & 1 & 1 & $a_6$ & & & $a_6=1,6$\\
1 & 1 & 1 & 3 & 3 & $a_6$ & & & $a_6=3,17,18,19$\\
1 & 1 & 3 & 3 & 3 & $a_6$ & & & $4\le a_6\le 12$, $a_6\neq 5,6,8$\\
1 & 1 & 3 & 3 & 4 & $a_6$ & & & $17\le a_6\le 19$\\
1 & 1 & 3 & 3 & 6 & $a_6$ & & & $a_6=6,9$ or $12\le a_6\le 15$\\
1 & 1 & 3 & 3 & 7 & $a_6$ & & & $7\le a_6\le 19$, $a_6\neq 8$\\
1 & 1 & 3 & 3 & 9 & $a_6$ & & & $9\le a_6\le 18$\\
1 & 1 & 3 & 6 & 6 & $a_6$ & & & $a_6=9$ or $12\le a_6\le 18$\\
1 & 1 & 3 & 6 & 9 & $a_6$ & & & $a_6=9$ or $12\le a_6\le 24$\\
1 & 1 & 3 & 6 & 12 & $a_6$ & & & $12\le a_6\le 24$\\
1 & 2 & 2 & 5 & 5 & $a_6$ & & & $a_6=5$ or $14\le a_6\le 18$\\
1 & 2 & 2 & 6 & 6 & $a_6$ & & & $a_6=6,14$ or $20\le a_6\le 25$\\
1 & 2 & 2 & 6 & 14 & $a_6$ & & & $a_6=14$ or $20\le a_6\le 39$\\
1 & 2 & 3 & 3 & 8 & $a_6$ & & & $12\le a_6\le 19$, $a_6\neq 13,14,16$\\
1 & 2 & 3 & 8 & 12 & $a_6$ & & & $12\le a_6\le 46$, $a_6\neq 13,14,16$\\
1 & 2 & 3 & 8 & 15 & $a_6$ & & & $15\le a_6\le 34$, $a_6\neq 16$\\
\hline 
1 & 1 & 3 & 3 & 3 & 3 & $a_7$ & & $a_7=6,13,14,15$\\
1 & 1 & 3 & 3 & 3 & 6 & $a_7$ & & $16\le a_7\le 18$\\
1 & 1 & 3 & 6 & 6 & 6 & $a_7$ & & $a_7=6$ or $19\le a_7\le 24$\\
\hline
1 & 1 & 3 & 3 & 3 & 3 & 3 & $a_8$ & $a_8=3,16,17,18$\\
\hline
\end{tabular}
\end{center}
\end{table}

%%%%%%%%%%%%%%%%%%%%%%%%%%%%%%%%%%%%%%%%%%%%%%%%%%%%%%%%%%%%%%%%%%%%%%%%%%%%%
\begin{table}[h]
\caption{Candidates for new universal hendecagonal forms $(\mathcal{GP}_{11},(a_1,a_2,\dots,a_k))$}
\label{table111}
\begin{center}
\begin{tabular}{|cccccccc|c|}
\hline
$a_1$ & $a_2$ & $a_3$ & $a_4$ & $a_5$ & $a_6$ & $a_7$ & $a_8$ & Conditions on $a_k$ ($4\le k\le 8$)\\
\hline
1 & 1 & 2 & $a_4$ & & & & & $a_4=3,4$\\
1 & 2 & 2 & 4 & & & & & \\
1 & 2 & 3 & 4 & & & & & \\
1 & 2 & 4 & $a_4$ & & & & & $4\le a_4\le 8$\\
\hline
1 & 1 & 1 & 1 & $a_5$ & & & & $a_5=3,4,5$\\
1 & 1 & 1 & 2 & $a_5$ & & & & $a_5=2,5,6$\\
1 & 1 & 1 & 3 & $a_5$ & & & & $4\le a_5\le 7$\\
1 & 1 & 1 & 4 & $a_5$ & & & & $4\le a_5\le 18$\\
1 & 1 & 2 & 2 & $a_5$ & & & & $a_5=2,5,6,7$\\
1 & 1 & 2 & 5 & $a_5$ & & & & $5\le a_5\le 20$\\
1 & 1 & 3 & 3 & $a_5$ & & & & $a_5=4,5,6,9$\\
1 & 1 & 3 & 4 & $a_5$ & & & & $a_5=5,8,9$\\
1 & 1 & 3 & 5 & $a_5$ & & & & $6\le a_5\le 18$\\
1 & 1 & 3 & 6 & $a_5$ & & & & $6\le a_5\le 13$, $a_5\neq 10$\\
1 & 2 & 2 & 2 & $a_5$ & & & & $2\le a_5\le 9$, $a_5\neq 4$\\
1 & 2 & 2 & 3 & $a_5$ & & & & $3\le a_5\le 9$, $a_5\neq 4$\\
1 & 2 & 2 & 5 & $a_5$ & & & & $5\le a_5\le 14$\\
1 & 2 & 2 & 6 & $a_5$ & & & & $6\le a_5\le 20$, $a_5\neq 17$\\
1 & 2 & 3 & 3 & $a_5$ & & & & $5\le a_5\le 12$, $a_5\neq 6,9$\\
1 & 2 & 3 & 5 & $a_5$ & & & & $5\le a_5\le 12$\\
1 & 2 & 3 & 6 & $a_5$ & & & & $7\le a_5\le 15$\\
1 & 2 & 3 & 7 & $a_5$ & & & & $8\le a_5\le 38$\\
1 & 2 & 4 & 9 & $a_5$ & & & & $9\le a_5\le 18$\\
\hline
1 & 1 & 1 & 1 & 1 & $a_6$ & & & $a_6=2,6$\\
1 & 1 & 1 & 1 & 2 & 7 & & & \\
1 & 1 & 1 & 3 & 3 & $a_6$ & & & $a_6=3,8$ or $10\le a_6\le 21$\\
1 & 1 & 3 & 3 & 3 & $a_6$ & & & $a_6=3,7,8,11,12,13$\\
1 & 1 & 3 & 3 & 7 & $a_6$ & & & $7\le a_6\le 20$, $a_6\neq 9$\\
1 & 1 & 3 & 3 & 8 & $a_6$ & & & $8\le a_6\le 21$, $a_6\neq 9$\\
1 & 1 & 3 & 3 & 10 & $a_6$ & & & $10\le a_6\le 20$\\
1 & 1 & 3 & 4 & 4 & $a_6$ & & & $4\le a_6\le 21$, $a_6\neq 5,8,9$\\
1 & 1 & 3 & 4 & 6 & $a_6$ & & & $a_6=10$ or $14\le a_6\le 27$\\
1 & 1 & 3 & 4 & 7 & $a_6$ & & & $a_6=7$ or $10\le a_6\le 17$\\
1 & 1 & 3 & 4 & 10 & $a_6$ & & & $10\le a_6\le 27$\\
1 & 1 & 3 & 5 & 5 & $a_6$ & & & $a_6=5$ or $19\le a_6\le 23$\\
1 & 1 & 3 & 6 & 10 & $a_6$ & & & $a_6=10$ or $14\le a_6\le 23$\\
1 & 2 & 2 & 6 & 17 & $a_6$ & & & $a_6=17$ or $21\le a_6\le 37$\\
1 & 2 & 3 & 3 & 3 & $a_6$ & & & $a_6=9,13,14,15$\\
1 & 2 & 3 & 3 & 6 & $a_6$ & & & $a_6=6,16,17,18$\\
1 & 2 & 3 & 3 & 9 & $a_6$ & & & $a_6=9$ or  $13\le a_6\le 21$\\
1 & 2 & 3 & 6 & 6 & $a_6$ & & & $a_6=6$ or  $16\le a_6\le 21$\\
1 & 2 & 3 & 7 & 7 & $a_6$ & & & $a_6=7$ or  $39\le a_6\le 45$\\
\hline 
1 & 1 & 1 & 1 & 1 & 1 & $a_7$ & & $a_7=1,7$\\
1 & 1 & 3 & 3 & 3 & 10 & $a_7$ & & $21\le a_7\le 23$\\
1 & 2 & 3 & 3 & 3 & 3 & $a_7$ & & $a_7=6,16,17,18$\\
1 & 2 & 3 & 3 & 3 & 6 & $a_7$ & & $19\le a_7\le 21$\\
\hline
1 & 2 & 3 & 3 & 3 & 3 & 3 & $a_8$ & $a_8=3,19,20,21$\\
\hline
\end{tabular}
\end{center}
\end{table}

\clearpage
%%%%%%%%%%%%%%%%%%%%%%%%%%%%%%%%%%%%%%%%%%%%%%%%%%%%%%%%%%%%%%%%%%%%%%%%%%%%%

\end{document}